\documentclass[11pt]{article}
       \usepackage{amsfonts}
       \usepackage{stmaryrd}
       \usepackage{latexsym,amssymb,mathrsfs,fancyhdr}
       \font\tenmsb=msbm10
       \font\sevenmsb=msbm7
       \font\fivemsb=msbm5
       \catcode`\@=11
       \ifx\amstexloaded@\relax\catcode`\@=\active
       \endinput\else\let\amstexloaded@\relax\fi
       \def\spaces@{\space\space\space\space\space}
       \def\spaces@@{\spaces@\spaces@\spaces@\spaces@\spaces@}
       \def\space@.  {\futurelet\space@\relax}
       \space@.   %
       \def\Err@#1{\errhelp\defaulthelp@\errmessage{AmS-TeX error: #1}}
       \def\relaxnext@{\let\next\relax}
       \def\accentfam@{7}
       \def\noaccents@{\def\accentfam@{0}}
       \def\Cal{\relaxnext@\ifmmode\let\next\Cal@\else
       \def\next{\Err@{Use \string\Cal\space only in math mode}}\fi\next}
       \def\Cal@#1{{\Cal@@{#1}}}
       \def\Cal@@#1{\noaccents@\fam\tw@#1}
       \def\Bbb{\relaxnext@\ifmmode\let\next\Bbb@\else
       \def\next{\Err@{Use \string\Bbb\space only in math mode}}\fi\next}
       \def\Bbb@#1{{\Bbb@@{#1}}}
       \def\Bbb@@#1{\noaccents@\fam\msbfam#1}
       \newfam\msbfam
       \textfont\msbfam=\tenmsb
       \scriptfont\msbfam=\sevenmsb
       \scriptscriptfont\msbfam=\fivemsb
\usepackage{dsfont}
\usepackage[german,english]{babel}
\usepackage{amsmath,amssymb}
\usepackage[square, comma, sort&compress, numbers]{natbib}
\usepackage{cases}
\usepackage{mathrsfs}
\usepackage{amsmath}
\usepackage{amsthm}
\usepackage{amsfonts}
\usepackage{amssymb}
\usepackage{latexsym}
\usepackage{fancyhdr}
\usepackage{extarrows}
\usepackage{geometry}
\usepackage{color}
\usepackage[comma,numbers,square,sort&compress]{natbib}
\usepackage{epstopdf}
\usepackage{graphicx,amsmath}
\usepackage{subfigure}
\usepackage{amssymb}
\theoremstyle{plain}
\newtheorem{theorem}{Theorem}[section]

\newtheorem{lemma}[theorem]{Lemma}
\newtheorem{example}[theorem]{Example}

\numberwithin{equation}{section}
\newtheorem{remark}[theorem]{Remark}

\theoremstyle{definition}
\newtheorem{definition}[theorem]{Definition}

\large\normalsize

\usepackage{enumerate}

\begin{document}
\numberwithin{equation}{section}
\title{{\bf The $\mathfrak{m}$-DMP inverse in Minkowski  space \\  and its applications}}

\author{{\ Jiale Gao $^{a}$, Qingwen Wang $^{b}$, Kezheng Zuo $^{a}$\footnote{E-mail address: xiangzuo28@163.com (K. Zuo).}}
\\{{\small $^a$ School of Mathematics and Statistics, Hubei Normal University, Huangshi, 435002, PR China}}
\\ {\small $^b$ College of Science, Shanghai University, Shanghai, 200444, PR China}}

\maketitle

\begin{center}
\begin{minipage}{135mm}
{{\small {\bf Abstract:}  This paper first  introduces a new generalized inverse in Minkowski space, called the $\mathfrak{m}$-DMP inverse, and  discusses  its algebraic and geometrical properties.  The second objective is to characterize the $\mathfrak{m}$-DMP inverse   equivalently by
ranges, null spaces and matrix equations, and show  its integral and limiting representations and several explicit expressions. Finally, the  paper gives  applications of the $\mathfrak{m}$-DMP inverse in solving a system of linear equations and a constrained optimization problem.
\begin{description}
\item[{\bf Key words}:]  $\mathfrak{m}$-DMP inverse; Minkowski  space; Hartwig-Spindelb\"{o}ck decomposition; full-rank factorization; system of linear equations
\item[{\bf AMS subject classifications}:]15A09, 15A03, 15A24
\end{description}
}}
\end{minipage}
\end{center}

\section{Introduction}\label{introductionsection}
 Since Malik and Thome \cite{DMPorignre} defined the DMP inverse   by using  the Moore-Penrose inverse \cite{penrosere} and the Drazin inverse \cite{Drainzeinre},  there has been tremendous interest  in  developing the DMP inverse in recent years. Liu and Cai \cite{DMPfutre8} proposed two iterative methods to compute the DMP inverse. The integral  and  determinantal representations for the DMP inverse were  derived by  \cite{DMPfutre10} and \cite{DMPfutre15}, respectively. Ferreyra et al. \cite{mixclasschagire}  developed the maximal classes of matrices  to determine  the DMP inverse. Using the  classical Cayley-Hamilton theorem, Wang et al. \cite{DMPfutre12} gave an annihilating polynomial of the DMP inverse. Ma et al. \cite{chaDMPmare}  investigated characterizations, iterative methods, sign patterns and perturbation analysis for the DMP inverse as well as  its applications in solving singular linear systems. Zuo et al. \cite{diffDMPre} presented further characterizations of the DMP inverse in terms of its range  and null space.  Furthermore,  the notion of the DMP inverse was extended from  square complex matrices   to  rectangular complex  matrices \cite{DMPfutre7}, operators in Hilbert spaces \cite{DMPfutre3},  elements in rings \cite{DMPfutre13},  finite potent endomorphisms  on arbitrary vector spaces \cite{DMPfutre16},  square matrices  over the quaternion skew  field \cite{DMPfutre19}, and tensors \cite{DMPfutre25}. And,   other extended forms of the DMP inverse were established by \cite{DMPfutre21,DMPfutre24}.
\par
 In studying polarized hight, Renardy \cite{minSVDre} investigated the singular value decomposition in Minkowski space in order to quickly verify that a Mueller matrix map the forward light cone into itself. Subsequently,
the Minkowski inverse  in Minkowski space  was   established by Meenakshi \cite{minMPre}, who also  gave a condition for a Mueller matrix to have a singular value decomposition in Minkowski space according to its Minkowski inverse.
In the past two decades, a great deal of mathematical effort   has been devoted to the study of the Minkowski inverse.
More details of its properties, applications and generalizations  can be found in \cite{algebraicMMP,minEPre,minMPEPre,appequre,partordmEPre,WminMPorire,WminMPiterre,WminMPfurre,indef}.
 Recently, Wang et al. defined the  $\mathfrak{m}$-core inverse \cite{MCore}, $\mathfrak{m}$-core-EP inverse \cite{mincoreEPre} and  $\mathfrak{m}$-WG inverse \cite{minwgroupre} in Minkowski space, which can be regarded as extensions of the core inverse \cite{Coreinverse},  core-EP inverse \cite{coreEPinversere} and   weak group inverse \cite{weakgroupinversere}, respectively.
\par
Inspired  by the study of the DMP inverse and the generalized inverses in Minkowski space,  the intention of this paper is  to
  introduce a new   generalized inverse in Minkowski space, called the $\mathfrak{m}$-DMP inverse, and    discuss its properties, characterizations, representations and applications.
\par
The primary contributions of the paper are summed up as follows:
\begin{itemize}
  \item   The definition of the $\mathfrak{m}$-DMP inverse in Minkowski space is given as the unique solution of a certain system of  matrix equations.  Based on its explicit expression, the canonical  forms of the $\mathfrak{m}$-DMP inverse is also obtained in terms of the  Hartwig-Spindelb\"{o}ck decomposition.

  \item The $\mathfrak{m}$-DMP inverse is represented as an outer inverse with prescribed range and null space, and some of its algebraic and  geometrical  properties are shown. On the converse, the $\mathfrak{m}$-DMP inverse is
      equivalently characterized by using its basic properties.

 \item Applying the full-rank factorization leads to an
explicit formula of  the $\mathfrak{m}$-DMP inverse. According to this result, we present an integral representation of the $\mathfrak{m}$-DMP inverse.  And, a few limiting representations of the $\mathfrak{m}$-DMP inverse are proposed.

\item  We apply the $\mathfrak{m}$-DMP inverse to solve a system of linear equations in Minkowski space as well as a least norm problem. And, a condensed Cramer's rule for the unique solution of this system is stated.
\end{itemize}
\par
The present paper is built up as follows.   Some necessary notions, definitions and lemmas are recalled in Section \ref{preliminariessec}.
 Section \ref{deminDMPsec} is devoted to introduce the $\mathfrak{m}$-DMP inverse and its properties. Further characterizations and representations of the $\mathfrak{m}$-DMP inverse are shown in Section \ref{charepminDMPsec}. Section \ref{appminDMPsec} presents applications of the $\mathfrak{m}$-DMP inverse in solving a system of linear equations and an optimization problem. The conclusion is stated in Section \ref{conclusionsec}.

\section{Preliminaries}\label{preliminariessec}
We use the following notations  throughout this paper.  Let $\mathbb{C}^{n}$, $\mathbb{C}^{m\times n}$ and $\mathbb{C}^{n\times n}_k$ be the sets of all complex $n$-dimensional vectors, complex  ${m\times n}$ matrices, and complex  ${n\times n}$ matrices with index $k$, respectively.    The smallest nonnegative integer $k$ satisfying ${\rm rank}(A^{k+1})={\rm rank}(A^{k})$ is called the index of $A\in\mathbb{C}^{n\times n}$, denoted by ${\rm Ind}(A)$.    The symbols $A^*$,  ${\rm rank}( A )$, $\mathcal{R}(A)$, $\mathcal{N}(A)$, and $\Vert A \Vert _F$ stand for the conjugate transpose, rank, range, null space, and Frobenius norm of $A\in\mathbb{C}^{ m \times n}$, respectively. We denote the identity matrix in $\mathbb{C}^{n\times n}$ by $I_n$, and the null matrix with appropriate orders by $0$.   The projector  onto $\mathcal{S}$ along $\mathcal{T}$ is indicated by  $P_{\mathcal{S},\mathcal{T}}$, where $\mathcal{S},\mathcal{T}\subseteq\mathbb{C}^{n}$   are subspaces satisfying that their  direct sum is $\mathbb{C}^{n}$, i.e., $\mathcal{S}\oplus\mathcal{T}=\mathbb{C}^{n}$.
\par
Additionally, the Minkowski inner product  \cite{minSVDre,minMPre} of two elements $x$ and $y$ in $\mathbb{C}^{n}$  is defined by $(x,y)=<x,Gy>$,
where
$
G=\left(
    \begin{array}{cc}
      1 & 0 \\
      0 & -I_{n-1} \\
    \end{array}
  \right)
$ represents the Minkowski metric matrix with order $n$, and  $<\cdot,\cdot>$ is the conventional Euclidean inner product. The complex linear space $\mathbb{C}^{n}$   with the Minkowski inner product is called  the  Minkowski space. Note that the Minkowski space is also an indefinite inner product space \cite{indef}. The Minkowski adjoint of $A\in\mathbb{C}^{m\times n}$ is $A^{\sim}=GA^*F$, where $G$ and $F$ are Minkowski metric matrices with orders $n$ and $m$, respectively.
\par
 Next, we will review definitions of some generalized inverses.
\begin{definition}\cite{penrosere,wangbookre}
Let $A\in\mathbb{C}^{m\times n}$. Then the matrix $X\in\mathbb{C}^{n\times m}$ verifying
\begin{equation*}
\left.
  \begin{array}{cccc}
      AXA=A, &  XAX=X, & (AX)^*=AX,  & (XA)^*=XA, \\
  \end{array}
\right.
\end{equation*}
 is called the Moore-Penrose inverse  of $A$, denoted by $A^{\dag}$.
 In addition,  if $X$ satisfies $XAX=X$,
then we  call $X$  an  outer inverse  of $A$.  For subspaces $\mathcal{T}\subseteq\mathbb{C}^n$ and $\mathcal{S}\subseteq\mathbb{C}^m$,
an outer inverse $X$ of $A$ with  $\mathcal{R}(X)=\mathcal{T}$ and $\mathcal{N}(X)=\mathcal{S}$ is unique, and is denoted  by $A_{\mathcal{T},\mathcal{S}}^{(2)}$.
\end{definition}

\begin{definition}\cite{Drainzeinre,grouporgre}\label{Drazininvdef}
Let $A\in\mathbb{C}^{n\times n}_k$. Then the matrix $X\in\mathbb{C}^{n\times n}$ satisfying  \begin{equation*}
\left.
  \begin{array}{cccc}
     XAX=X,   &   AX=XA,   &   A^{k+1}X=A^{k}, \\
  \end{array}
\right.
\end{equation*}
 is called the Drazin inverse  of $A$,  denoted by $A^{D}$. In the case ${\rm Ind}(A)=1$, the Drazin inverse  of $A$ reduces the group inverse of $A$, which is denoted by $A^{\#}$.
\end{definition}

\begin{definition}\cite{DMPorignre}
Let $A\in\mathbb{C}^{n\times n}_k$. Then we call the matrix $X\in\mathbb{C}^{n\times n}$ fulfilling
\begin{equation*}
\left.
  \begin{array}{cccc}
     XAX=X,   & XA=A^{D}A,    &   A^{k}X=A^{k}A^{\dag},  \\
  \end{array}
\right.
\end{equation*}
the DMP inverse of $A$, which is  denoted by $A^{D,\dag}$. And, $A^{D,\dag}=A^DAA^{\dag}$.
\end{definition}

\begin{definition}\cite{minMPre}\label{minMPdefintion}
Let $A\in\mathbb{C}^{m\times n}$.
If there exists a matrix $X\in\mathbb{C}^{n\times m}$ such that
\begin{equation*}
\left.
  \begin{array}{cccc}
    AXA=A,  &   XAX=X,   &  (AX)^{\sim}=AX,  &  (XA)^{\sim}=XA,\\
  \end{array}
\right.
\end{equation*}
then $X$ is called the Minkowski inverse of $A$, denoted by $A^{\mathfrak{m}}$.
\end{definition}

Subsequently, we recall a few  auxiliary  lemmas which will be utilized later. First off, we mention  the Hartwig-Spindelb\"{o}ck decomposition as an effective tool in studying  generalized inverses.
\begin{lemma}[Hartwig-Spindelb\"{o}ck decomposition, \cite{HSdec}]\label{HSth}
Let $A\in\mathbb{C}^{n\times n}$ and $r={\rm rank}(A)$. Then $A$ can be represented in the from
\begin{equation}\label{HSAdec}
A=U\left(
     \begin{array}{cc}
       \Sigma K & \Sigma L \\ 0 & 0 \\
     \end{array}
   \right)U^*,
\end{equation}
where $U\in\mathbb{C}^{n\times n}$ is unitary,
$\Sigma ={\rm diag}$ $(\sigma_1,\sigma_2,...,\sigma_r)$ is the diagonal matrix of singular values of $A$, $\sigma_i>0$ $(i=1,2,...,r)$, and $K\in\mathbb{C}^{r\times r}$ and $L\in\mathbb{C}^{r\times (n-r)}$ satisfy $KK^*+LL^*=I_r$.
\end{lemma}

\begin{lemma}\cite[Formula (14)]{DMPorignre}
Let $A\in\mathbb{C}^{n\times n}$ be given by \eqref{HSAdec}. Then,
\begin{equation}\label{Ddec01}
    A^D=U
   \left(
     \begin{array}{cc}
       (\Sigma K)^D & \left((\Sigma K)^D\right)^2\Sigma L \\
       0 & 0 \\
     \end{array}
   \right)
    U^*.
\end{equation}

\end{lemma}

 Several  significant properties of the Drazin inverse and the Minkowski inverse are referenced.

\begin{lemma}\label{HSmMpdecth}
Let $A\in\mathbb{C}^{m\times n}$. Then the following statements are equivalent:
\begin{enumerate}[$(1)$]
\item $A^{\mathfrak{m}}$ exists;
  \item\label{minMPexsitconditem2} \cite[Theorem 1]{minMPre} ${\rm rank}(AA^{\sim})={\rm rank}(A^{\sim}A)={\rm rank}(A)$;
        \item\label{minMPexsitconditem3} ${\rm rank}({A^{\sim}}AA^{\sim})={\rm rank}(A)$;
\item\cite[Theorems 7,8]{algebraicMMP}\label{minMPexsitconditem4}
$CC^{\sim}$ and $B^{\sim}B$ are nonsingular, where $BC=A$ is a full rank factorization of $A$ with rank $r$, in which case,
$
A^{\mathfrak{m}}=C^{\sim}(CC^{\sim})^{-1}(B^{\sim}B)^{-1}B^{\sim}
$.
\end{enumerate}
\end{lemma}

\begin{proof}
We only prove  \eqref{minMPexsitconditem2} $\Leftrightarrow$ \eqref{minMPexsitconditem3}.  In fact, since $\mathcal{R}(A)\cap \mathcal{N}(A^{\sim})=\{0\}$ by ${\rm rank}(A^{\sim}A)={\rm rank}(A)$, it follows from ${\rm rank}(AA^{\sim})={\rm rank}(A)$ that
\begin{align*}
{\rm rank}({A^{\sim}}AA^{\sim})&={\rm rank}(AA^{\sim})-{\rm dim}(\mathcal{R}(AA^{\sim}) \cap \mathcal{N}(A^{\sim}))\\
&={\rm rank}(A)-{\rm dim}(\mathcal{R}(A)\cap \mathcal{N}(A^{\sim})) ={\rm rank}(A).
\end{align*}
Conversely, it is easy and is therefore omitted.
\end{proof}

\begin{lemma}\label{mMPprole}\cite[Theorem 9]{indef}
Let $A\in\mathbb{C}^{m\times n}$ be such that $A^{\mathfrak{m}}$ exists.
Then,
\begin{enumerate}[$(1)$]
\item\label{mMPRN} $\mathcal{R}(A^{\mathfrak{m}})=\mathcal{R}(A^{\sim})$ and $\mathcal{N}(A^{\mathfrak{m}})=\mathcal{N}(A^{\sim})$;
  \item\label{AAmMPeq} $AA^{\mathfrak{m}}=P_{\mathcal{R}(A),\mathcal{N}(A^{\sim})}$;
  \item\label{AmAMPeq} $A^{\mathfrak{m}}A=P_{\mathcal{R}(A^{\sim}),\mathcal{N}(A)}$.
\end{enumerate}
\end{lemma}

\begin{lemma}\cite[Theorem 2.1.4]{wangbookre}\label{Dinverseperle}
Let  $A\in\mathbb{C}^{n\times n}_{k}$.  Then,
\begin{enumerate}[$(1)$]
  \item\label{Dinverseperitem1} $  \mathcal{R}(A^D)=\mathcal{R}(A^{k})$ and $ \mathcal{N}(A^D)=\mathcal{N}(A^{k}) $;
  \item\label{Dinverseperitem2} $ AA^{D}=A^DA= P_{\mathcal{R}(A^{k}),\mathcal{N}(A^{k})}$.
\end{enumerate}

\end{lemma}

And, a few limiting expressions of an outer inverse with prescribed range and null space are reviewed.

\begin{lemma}\cite[Theorem 2.1]{ATSlim}\label{limitlemma1}
Let  $A\in\mathbb{C}^{m\times n}$, $X \in\mathbb{C}^{n\times p}$ and $Y\in\mathbb{C}^{p\times m}$. If $A^{(2)}_{\mathcal{R}(XY),\mathcal{N}(XY)}$ exists, then
\begin{equation}\label{ATS2YZlimteq}
A^{(2)}_{\mathcal{R}(XY),\mathcal{N}(XY)}= {\rm lim}_{\lambda \to 0}X(\lambda I_p+YAX)^{-1}Y.
\end{equation}
\end{lemma}

\begin{lemma}\label{ATSlimitlemma}\cite[Theorem 2.4]{ATSlimit1}
 Let $A\in\mathbb{C}^{m\times n}$, and let $H\in\mathbb{C}^{n\times m}$ be such that $\mathcal{R}(H)=\mathcal{T}$  and $\mathcal{N}(H)=\mathcal{S}$, where $\mathcal{T}$  and $\mathcal{S}$ are subspaces of $\mathbb{C}^{n}$ and $\mathbb{C}^{m}$,  respectively. If $A_{\mathcal{T},\mathcal{S}}^{(2)}$ exists, then
\begin{align}
A_{\mathcal{T},\mathcal{S}}^{(2)}
&={\rm lim}_{\lambda \to 0}H(\lambda I_m+AH)^{-1}\label{ATS2Ilimiteq02}\\
&={\rm lim}_{\lambda \to 0}(\lambda I_n+HA)^{-1}H.\label{ATS2Ilimiteq03}
\end{align}
\end{lemma}

\section{The $\mathfrak{m}$-DMP inverse in Minkowski Space}\label{deminDMPsec}
The main purpose of this section is to introduce the $\mathfrak{m}$-DMP inverse in Minkowski space, and present some of its properties. We  begin with considering the following system of matrix equations, whose unique solution is defined  as the $\mathfrak{m}$-DMP inverse.

\begin{theorem}\label{uniqueDMPeqth}
Let  $A\in\mathbb{C}^{n\times n}_{k}$ with ${\rm rank}(A^{\sim}AA^{\sim})={\rm rank}(A)$. Then the system of matrix equations
\begin{equation}\label{MDMPeq}
\left.
  \begin{array}{cccc}
  XAX=X,   &  XA=A^{D}A,  &  A^kX=A^kA^{\mathfrak{m}},  \\
  \end{array}
\right.
\end{equation}
has the unique solution
\begin{equation*}
X=A^{D}AA^{\mathfrak{m}}.
\end{equation*}
\end{theorem}

\begin{proof}
Using the condition \eqref{MDMPeq} and Lemma \ref{Dinverseperle}\eqref{Dinverseperitem2}, we have that
\begin{align*}
  X & =  XAX=A^DAX=(A^DA)^kX=(A^D)^kA^kA^{\mathfrak{m}}=A^DAA^{\mathfrak{m}},
\end{align*}
which completes the proof.
\end{proof}

\begin{definition}\label{minDMPdefinition}
Let  $A\in\mathbb{C}^{n\times n}_{k}$ with ${\rm rank}(A^{\sim}AA^{\sim})={\rm rank}(A)$. The   $\mathfrak{m}$-${\rm DMP}$ inverse of $A$  in Minkowski space, denoted by $A^{D,\mathfrak{m}}$, is defined as
\begin{equation} \label{MDMPeADAm}
  A^{D,\mathfrak{m}}=A^DAA^{\mathfrak{m}}.
\end{equation}

\end{definition}

\begin{remark}\label{mCoremDMPke1remark}
By comparing \cite[Theorem 2.9]{MCore} and   Definition  \ref{minDMPdefinition}, it is obvious  that the concept of the $\mathfrak{m}$-${\rm DMP}$ inverse  generalizes that of the  $\mathfrak{m}$-core inverse, which is denoted by $A^{\textcircled{m}}$. In other words,
if $A\in\mathbb{C}^{n\times n}_{1}$ satisfies ${\rm rank}(A^{\sim}AA^{\sim})={\rm rank}(A)$, then $A^{D,\mathfrak{m}}=A^{\textcircled{m}}$.
\end{remark}

 The following theorems give  canonical  representations of  the Minkowski inverse and the $\mathfrak{m}$-DMP inverse in terms of the  Hartwig-Spindelb\"{o}ck decomposition.

\begin{theorem}\label{HSmMpdecsecth}
Let $A\in\mathbb{C}^{n\times n}$ be given by \eqref{HSAdec}, let  $
 \Delta= \left(
            \begin{array}{cc}
              K & L \\
            \end{array}
          \right)
     U^*GU
         \left(
            \begin{array}{c}
              K^* \\ L^* \\
            \end{array}
          \right),
$
and let the partition of the Minkowski metric matrix $G\in\mathbb{C}^{n\times n}$ be
\begin{equation}\label{G1G2G3G4eq}
G=U \left(
           \begin{array}{cc}
             G_1 & G_2 \\
             G_2^* & G_4\\
           \end{array}
         \right) U^*,
 \end{equation}
 where $G_1\in\mathbb{C}^{r\times r}$, $G_2\in\mathbb{C}^{r\times (n-r)}$ and $G_4\in\mathbb{C}^{(n-r)\times (n-r)}$.
\begin{enumerate}[$(1)$]
  \item\label{HSmMpdecitem1} ${\rm rank}(A)={\rm rank}(AA^{\sim})$ if and only if $ \Delta$  is nonsingular.
  \item\label{HSmMpdecitem2} ${\rm rank}(A)={\rm rank}(A^{\sim}A)$ if and only if $G_1$ is nonsingular.
 \item\label{HSmMpdecitem3} If $ \Delta$  and $G_1$ are nonsingular, then
\begin{equation}\label{minMPdec01}
A^{\mathfrak{m}}=GU\left(
                     \begin{array}{cc}
                       K^*(G_1\Sigma\Delta)^{-1} & 0 \\
                       L^*(G_1\Sigma\Delta)^{-1} & 0 \\
                     \end{array}
                   \right)U^*G.
\end{equation}
\end{enumerate}
\end{theorem}

\begin{proof}
\eqref{HSmMpdecitem1}. Using the Hartwig-Spindelb\"{o}ck decomposition, we have \begin{align*}
{\rm rank}(A)={\rm rank}(AA^{\sim}) &\Leftrightarrow {\rm rank}(A)= {\rm rank}\left(
   \left(
     \begin{array}{cc}
       \Sigma K & \Sigma L \\
       0 & 0 \\
     \end{array}
   \right)
U^*GU
   \left(
     \begin{array}{cc}
       (\Sigma K)^* & 0 \\
       (\Sigma L)^* & 0 \\
     \end{array}
   \right)
   \right)\\
 &\Leftrightarrow  {\rm rank}(A)={\rm rank}\left(
 \left(
   \begin{array}{cc}
        K    &    L \\
   \end{array}
 \right)
U^*GU
 \left(
   \begin{array}{c}
    K^*  \\
     L^* \\
   \end{array}
 \right)
 \right),
\end{align*}
which is equivalent to that $\Delta$  is nonsingular.
\par
\eqref{HSmMpdecitem2}.
Since $\left(
  \begin{array}{cc}
    \Sigma K & \Sigma L\\
  \end{array}
\right)$ is of full row rank by $KK^*+LL^*=I_r$,
using again the Hartwig-Spindelb\"{o}ck decomposition we derive   that
\begin{align*}
{\rm rank}(A)={\rm rank}(A^{\sim}A) &\Leftrightarrow
{\rm rank}(A) ={\rm rank}\left(
   \left(
     \begin{array}{cc}
       (\Sigma K)^* & 0 \\
       (\Sigma L)^* & 0 \\
     \end{array}
   \right)\left(
           \begin{array}{cc}
             G_1 & G_2 \\
             G_2^* & G_4 \\
           \end{array}
         \right)
         \left(
     \begin{array}{cc}
       \Sigma K & \Sigma L \\
       0 & 0 \\
     \end{array}
   \right)
\right)\\
&\Leftrightarrow
{\rm rank}(A) ={\rm rank}\left(
\left(
  \begin{array}{cc}
    \Sigma K & \Sigma L\\
  \end{array}
\right)^*G_1
\left(
  \begin{array}{cc}
    \Sigma K & \Sigma L\\
  \end{array}
\right)
\right)\\
&\Leftrightarrow
{\rm rank}(A)={\rm rank}(G_1),
\end{align*}
which is equivalent to that $G_1$ is nonsingular.
\par
\eqref{HSmMpdecitem3}.
Note that $A$ given in \eqref{HSAdec} can be rewritten as
\begin{equation}\label{AHSdecrefullrankeq}
A=U\left(
     \begin{array}{c}
        \Sigma \\
       0 \\
     \end{array}
   \right)
   \left(\begin{array}{cc}
                K & L \\
            \end{array}
          \right)U^*,
\end{equation}
and
$B:=U\left(\begin{array}{c}
        \Sigma \\0 \\
     \end{array}
   \right)$
   and
$C:=\left(\begin{array}{cc}
                K & L \\
            \end{array}
          \right)U^*$
are of  full column rank and full  row rank, respectively.
If $ \Delta$  and $G_1$ are nonsingular, by Lemma \ref{HSmMpdecth}\eqref{minMPexsitconditem2} we see that $A^{\mathfrak{m}}$ exists.
Therefore, applying Lemma \ref{HSmMpdecth}\eqref{minMPexsitconditem4} to \eqref{AHSdecrefullrankeq} yields that
 \begin{align*}
  A^{\mathfrak{m}}= & C^{\sim}(CC^{\sim})^{-1}(B^{\sim}B)^{-1}B^{\sim}\\
    =& GU\left(
     \begin{array}{c}
        K^* \\
       L^* \\
     \end{array}
   \right)G\left(
   \left(\begin{array}{cc}
                K & L \\
            \end{array}
          \right)U^*GU\left(
     \begin{array}{c}
        K^* \\
       L^* \\
     \end{array}
   \right)G \right)^{-1}
   \\&
   \left(G\left(\begin{array}{cc}
               {\Sigma} & 0 \\
            \end{array}
          \right)U^*G
          U\left(\begin{array}{c}
        \Sigma \\0 \\
     \end{array}
   \right)
   \right)^{-1}G\left(\begin{array}{cc}
               {\Sigma} & 0 \\
            \end{array}
          \right)U^*G\\
  =&GU\left(
     \begin{array}{c}
        K^* \\
       L^* \\
     \end{array}
   \right){\Delta}^{-1}(\Sigma G_1 \Sigma)^{-1}\left(\begin{array}{cc}
               {\Sigma} & 0 \\
            \end{array}
          \right)U^*G\\
   =&GU\left(
                     \begin{array}{cc}
                       K^*(G_1\Sigma\Delta)^{-1} & 0 \\
                       L^*(G_1\Sigma\Delta)^{-1} & 0 \\
                     \end{array}
                   \right)U^*G,
 \end{align*}
 which completes the proof of this theorem.
\end{proof}

\begin{theorem}\label{MDMPHSdecth}
Let  $A\in\mathbb{C}^{n\times n}_{k}$  be given by \eqref{HSAdec} with ${\rm rank}(A^{\sim}AA^{\sim})={\rm rank}(A)$, and let $G_1$ and $G_2$ be given by \eqref{G1G2G3G4eq}. Then  $A^{D,\mathfrak{m}}$
 has the decompositions in the forms
 \begin{align}
A^{D,\mathfrak{m}}
&=
U\left(
  \begin{array}{cc}
  (\Sigma K)^{D}G_1^{-1} &0 \\
    0 & 0 \\
  \end{array}
\right)U^*G\label{minDMPHSdecGeq01}\\
&=
U\left(
  \begin{array}{cc}
  (\Sigma K)^{D} &(\Sigma K)^{D}G_1^{-1}G_2  \\
    0 & 0 \\
  \end{array}
\right)U^*.\label{minDMPHSdecGeq}
 \end{align}

\end{theorem}

\begin{proof}
Inserting \eqref{HSAdec}, \eqref{Ddec01} and \eqref{minMPdec01} to \eqref{MDMPeADAm},  by  direct calculation we infer that
\begin{align*}
A^{D,\mathfrak{m}}
=&U\left(
     \begin{array}{cc}
       (\Sigma K)^D & \left((\Sigma K)^D\right)^2\Sigma L \\
       0 & 0 \\
     \end{array}
   \right)\left(
     \begin{array}{cc}
       \Sigma K & \Sigma L \\ 0 & 0 \\
     \end{array}
   \right)U^* GU\left(
                     \begin{array}{cc}
                       K^*(G_1\Sigma\Delta)^{-1} & 0 \\
                       L^*(G_1\Sigma\Delta)^{-1} & 0 \\
                     \end{array}
                   \right)U^*G\\
   =& U\left(
     \begin{array}{cc}
        (\Sigma K)^D  & \left((\Sigma K)^D\right)^2\Sigma L \\
       0 & 0 \\
     \end{array}
   \right)\left(
         \begin{array}{c}
            \Sigma  \\
           0 \\
         \end{array}
       \right)\left(
         \begin{array}{cc}
      K &L\\
         \end{array}
       \right)
       U^*GU
       \left(
         \begin{array}{c}
            K^*  \\
           L^* \\
         \end{array}
       \right)
       \left(
         \begin{array}{cc}
      \Delta^{-1}\Sigma^{-1}G_1^{-1} & 0 \\
         \end{array}
       \right)U^*G\\
      =& U\left(
     \begin{array}{cc}
       (\Sigma K)^D  & \left((\Sigma K)^D\right)^2\Sigma L \\
       0 & 0 \\
     \end{array}
   \right)\left(
         \begin{array}{c}
            \Sigma  \\
           0 \\
         \end{array}
       \right)\Delta
       \left(
         \begin{array}{cc}
      \Delta^{-1}\Sigma^{-1}G_1^{-1} & 0 \\
         \end{array}
       \right)U^*G\\
=&U\left(
  \begin{array}{cc}
  (\Sigma K)^{D}G_1^{-1} &0 \\
    0 & 0 \\
  \end{array}
\right)U^*G =U\left(
  \begin{array}{cc}
  (\Sigma K)^{D}G_1^{-1} &0 \\
    0 & 0 \\
  \end{array}
\right)
  \left(
           \begin{array}{cc}
             G_1 & G_2 \\
             G_2^* & G_4 \\
           \end{array}
         \right)U^*\\
=&
U\left(
  \begin{array}{cc}
  (\Sigma K)^{D} &(\Sigma K)^{D}G_1^{-1}G_2  \\
    0 & 0 \\
  \end{array}
\right)U^*,
\end{align*}
which completes the proof.
\end{proof}

\begin{remark}
Under the hypotheses of Theorem \ref{MDMPHSdecth}, we have an another outer inverse associated with $A$, that is, $A^{\mathfrak{m},D}=A^{\mathfrak{m}}AA^{D}$, which is  called the dual $\mathfrak{m}$-${\rm DMP}$ inverse of $A$  in Minkowski space. Using  the method analogous to the proof  of \eqref{minDMPHSdecGeq01}, we have
\begin{equation*}
A^{\mathfrak{m},D}=GU
\left(
  \begin{array}{cc}
    K^*\Delta^{-1}K(\Sigma K)^{D} &  K^*\Delta^{-1}K((\Sigma K)^{D})^2\Sigma L \\
    L^*\Delta^{-1}K(\Sigma K)^{D} & L^*\Delta^{-1}K((\Sigma K)^{D})^2\Sigma L \\
  \end{array}
\right)U^*.
\end{equation*}
It is expected that $A^{\mathfrak{m},D}$ will have properties similar to that of $A^{D,\mathfrak{m}}$.
\end{remark}

\begin{example}\label{MinDMPmainExample}
 Let
\begin{equation*}
  A=\left(
  \begin{array}{ccccc}
1&0&0&0&0\\1&0&0&1&0\\0&0&0&1&0\\0&0&0&0&0\\0&0&0&0&0\\
  \end{array}
\right).
\end{equation*}
Then ${\rm rank}(A^{\sim}AA^{\sim})={\rm rank}(A)=2$,
\begin{align*}
  A^{\dag}&=\left(
  \begin{array}{ccccc}
0.66667&0.33333&-0.33333&0&0\\0&0&0&0&0\\0&0&0&0&0\\ -0.33333&0.33333&0.66667&0&0\\0&0&0&0&0\\
  \end{array}
\right),
A^{D}=\left(
  \begin{array}{ccccc}
1&0&0&0&0\\1&0&0&0&0\\0&0&0&0&0\\0&0&0&0&0\\0&0&0&0&0\\
  \end{array}
\right),\\
A^{D,\dag}&= \left(
  \begin{array}{ccccc}
0.66667&0.33333&-0.33333&0&0\\0.66667&0.33333&-0.33333&0&0 \\0&0&0&0&0\\0&0&0&0&0\\0&0&0&0&0\\
  \end{array}
\right),
A^{\mathfrak{m}}= \left(
  \begin{array}{ccccc}
2&-1&1&0&0\\0&0&0&0&0\\0&0&0&0&0\\-1&1&0&0&0\\0&0&0&0&0\\
  \end{array}
\right),\\
A^{D,\mathfrak{m}}&=\left(
  \begin{array}{ccccc}
2&-1&1&0&0\\2&-1&1&0&0\\0&0&0&0&0\\0&0&0&0&0\\0&0&0&0&0\\
  \end{array}
\right),
A^{\mathfrak{m},D}=\left(
  \begin{array}{ccccc}
1&0&0&0&0\\0&0&0&0&0\\0&0&0&0&0\\0&0&0&0&0\\0&0&0&0&0\\
  \end{array}
\right).
\end{align*}
Clearly, $A^{D,\mathfrak{m}}$ is a new generalized inverse of $A$, which is different from $A^{\dag}$, $A^{D}$, $A^{D,\dag}$ and $A^{\mathfrak{m}}$.
It  also shows that $A^{D,\mathfrak{m}}$ and $A^{\mathfrak{m},D}$ are different.
\end{example}

The following theorem gives  some basic properties of the $\mathfrak{m}$-DMP inverse, which show that the $\mathfrak{m}$-DMP inverse is an outer inverse with prescribed  range and null space.

\begin{theorem}\label{MDMPATSth}
Let  $A\in\mathbb{C}^{n\times n}_{k}$ with ${\rm rank}(A^{\sim}AA^{\sim})={\rm rank}(A)$.  Then,
\begin{enumerate}[$(1)$]
\item\label{minDMPperitem1} ${\rm rank}(A^{D,\mathfrak{m}})={\rm rank}(A^k)$;
\item \label{minDMPperitem2} $\mathcal{R}(A^{D,\mathfrak{m}})=\mathcal{R}(A^k)$ and $\mathcal{N}(A^{D,\mathfrak{m}})=\mathcal{N}(A^kA^{\mathfrak{m}})$;

\item\label{minDMPperitem3}
 $A^{D,\mathfrak{m}}=A^{(2)}_{\mathcal{R}(A^{k}),\mathcal{N}(A^kA^{\mathfrak{m}})}$;
  \item \label{minDMPperitem4} $AA^{D,\mathfrak{m}}=P_{\mathcal{R}(A^{k}),\mathcal{N}(A^kA^{\mathfrak{m}})}$;
  \item\label{minDMPperitem5} $A^{D,\mathfrak{m}}A=A^DA=P_{\mathcal{R}(A^{k}),\mathcal{N}(A^{k})}$.
\end{enumerate}
\end{theorem}

\begin{proof}
\eqref{minDMPperitem1}. It follows from Lemma \ref{Dinverseperle}   that
  \begin{equation*}
  {\rm rank}(A^D) ={\rm rank}(A^DA)={\rm rank}(A^DAA^{\mathfrak{m}}A) \leq{\rm rank}(A^DAA^{\mathfrak{m}})\leq {\rm rank}(A^D),
  \end{equation*}
  which, together with \eqref{MDMPeADAm},  shows that ${\rm rank}(A^{D,\mathfrak{m}})={\rm rank}(A^k)$.\par

\eqref{minDMPperitem2}. Using  \eqref{MDMPeADAm} and the item \eqref{minDMPperitem1},  we have $\mathcal{R}(A^{D,\mathfrak{m}})=\mathcal{R}(A^k)$ directly.
Since
\begin{equation*}
 {\rm rank}(A^{k})={\rm rank}(A^kA^{\mathfrak{m}}A)\leq{\rm rank}(A^kA^{\mathfrak{m}})\leq {\rm rank}(A^k),
\end{equation*}
again by the item \eqref{minDMPperitem1} we get
\begin{equation}\label{rankAKAmerankAkeq}
  {\rm rank}(A^kA^{\mathfrak{m}})=  {\rm rank}(A^{k})= {\rm rank}(A^{D,\mathfrak{m}}).
\end{equation}
Moreover,
\begin{equation*}
 \mathcal{N}(A^kA^{\mathfrak{m}})\subseteq   \mathcal{N}((A^DA)^kA^{\mathfrak{m}})= \mathcal{N}(A^DAA^{\mathfrak{m}})=\mathcal{N}(A^{D,\mathfrak{m}}),
\end{equation*}
implying  $\mathcal{N}(A^{D,\mathfrak{m}})= \mathcal{N}(A^kA^{\mathfrak{m}})$.
\par
\eqref{minDMPperitem3}. It is obvious   by   \eqref{MDMPeq} and the item \eqref{minDMPperitem2}.
\par
\eqref{minDMPperitem4}.
In terms of ${\rm Ind}(A)=k$ and the item \eqref{minDMPperitem2}, we infer that
\begin{equation*}
\mathcal{R}(AA^{D,\mathfrak{m}})= A\mathcal{R}(A^{k})=\mathcal{R}(A^{k+1})=\mathcal{R}(A^{k}).
\end{equation*}
Evidently, ${\rm rank}(AA^{D,\mathfrak{m}})={\rm rank}(A^{k})$, which, together with the items \eqref{minDMPperitem1}  and  \eqref{minDMPperitem2},
shows that
\begin{equation*}
  \mathcal{N}(AA^{D,\mathfrak{m}})=\mathcal{N}(A^{D,\mathfrak{m}})=\mathcal{N}(A^kA^{\mathfrak{m}}).
\end{equation*}
Then, since  $A^{D,\mathfrak{m}}$ is an outer inverse of $A$,  we have   $AA^{D,\mathfrak{m}}=P_{\mathcal{R}(A^{k}), \mathcal{N}(A^kA^{\mathfrak{m}})}$.
\par
 \eqref{minDMPperitem5}. It is easily  obtained  by \eqref{MDMPeADAm} and  Lemma \ref{Dinverseperle}\eqref{Dinverseperitem2}.
\end{proof}

It is a popular approach to  characterize  generalized inverses  from the geometric point of view,  for example,  \cite[Definiton 1]{Coreinverse},  \cite[Theorem 2.13]{DMPorignre}, \cite[Theorem 3.2]{DMPfutre7},  etc. So, the following theorem  presents a geometric characterization of the $\mathfrak{m}$-DMP inverse.

\begin{theorem}\label{MDMPgeoth}
Let  $A\in\mathbb{C}^{n\times n}_{k}$ with ${\rm rank}(A^{\sim}AA^{\sim})={\rm rank}(A)$. Then $A^{D,\mathfrak{m}}$ is the unique matrix $X\in\mathbb{C}^{n\times n}$ such that
\begin{equation}\label{MDMPgeotheq}
\left.
  \begin{array}{cccc}
     AX=P_{\mathcal{R}(A^{k}),\mathcal{N}(A^kA^{\mathfrak{m}})},    &     \mathcal{R}(X)\subseteq \mathcal{R}({A^k}).  \\
  \end{array}
\right.
\end{equation}
\end{theorem}

\begin{proof}
Obviously, from   Theorem \ref{MDMPATSth}\eqref{minDMPperitem2} and \eqref{minDMPperitem4}, we see that $A^{D,\mathfrak{m}}$ is a solution to \eqref{MDMPgeotheq}.  Then, we will prove the uniqueness of the solution of  \eqref{MDMPgeotheq}.
Assume that both $X_1$ and $X_2$ are such that \eqref{MDMPgeotheq}. Thus, since $AX_1=AX_2=P_{\mathcal{R}(A^{k}),\mathcal{N}(A^kA^{\mathfrak{m}})}$, we get $\mathcal{R}(X_1-X_2)\subseteq \mathcal{N}(A)\subseteq\mathcal{N}(A^k)$.
Moreover, it follows from $\mathcal{R}(X_1)\subseteq \mathcal{R}({A^k})$ and $\mathcal{R}(X_2)\subseteq \mathcal{R}({A^k})$ that $\mathcal{R}(X_1-X_2)\subseteq  \mathcal{R}(A^k)$. Hence,  according to  ${\rm Ind}(A)=k$, we directly obtain that
\begin{equation*}
\mathcal{R}(X_1-X_2)\subseteq  \mathcal{R}(A^k)\cap\mathcal{N}(A^k)=\{0\},
\end{equation*}
i.e.,  $X_1=X_2$. Therefore, $A^{D,\mathfrak{m}}$ is the unique solution to \eqref{MDMPgeotheq}.
\end{proof}

\begin{remark}
Let  $A\in\mathbb{C}^{n\times n}_1$ with ${\rm rank}(A^{\sim}AA^{\sim})={\rm rank}(A)$. In terms of \cite[Theorem 2.7(I) and Theorem 2.9]{MCore} and Lemma  \ref{mMPprole}\eqref{mMPRN}, it can easily be  obtained that $A^{\textcircled{m}}=A^{(1,2)}_{\mathcal{R}(A),\mathcal{N}(A^{\sim})}$. And, in terms of Remark \ref{mCoremDMPke1remark}, it is a direct corollary of Theorem \ref{MDMPgeoth} that $A^{\textcircled{m}}$ is the unique matrix $X\in\mathbb{C}^{n\times n}$ such that
$AX=P_{\mathcal{R}(A),\mathcal{N}(A^{\sim})}$ and $\mathcal{R}(X)\subseteq \mathcal{R}({A})$.
\end{remark}

\par
The following theorem shows some new properties of the $\mathfrak{m}$-DMP inverse, which inherit from that of the DMP  inverse \cite[Proposition 2.14]{DMPorignre}.

\begin{theorem}\label{mDMPfurperth}
Let  $A\in\mathbb{C}^{n\times n}_{k}$  be given by \eqref{HSAdec} with ${\rm rank}(A^{\sim}AA^{\sim})={\rm rank}(A)$. Then,
\begin{enumerate}[$(1)$]
  \item\label{mDMPfurperitem1} $A^{D,\mathfrak{m}}=A^{D}P_{\mathcal{R}(A),\mathcal{N}(A^{\sim})}$;
  \item\label{mDMPfurperitem2}  $(A^{D,\mathfrak{m}})^l
  =\left\{\begin{array}{ll}
              (A^DA^{\mathfrak{m}})^{\frac{l}{2}}, &\text{if } l  \text{ is even,}\\
              A(A^DA^{\mathfrak{m}})^{\frac{l+1}{2}}, &\text{if } l  \text{ is odd;} \\
            \end{array}\right.$
  \item\label{mDMPfurperitem3}  $A^{D,\mathfrak{m}}= (A^2A^{\mathfrak{m}})^{D}$;
  \item\label{mDMPfurperitem4}  $((A^{D,\mathfrak{m}})^D)^D=A^{D,\mathfrak{m}}$;
  \item\label{mDMPfurperitem7}
  $AA^{D,\mathfrak{m}}=A^{D,\mathfrak{m}}A$ if and only if
  $A^{D,\mathfrak{m}}=A^{D}$   if and only if  $\mathcal{N}(A^{\sim})\subseteq \mathcal{N}(A^{k})$;
  \item\label{mDMPfurperitem5}  $A^{D,\mathfrak{m}}=0$ if and only if $A$ is nilpotent;
  \item\label{mDMPfurperitem6}  $A^{D,\mathfrak{m}}=A$ if and only if $(\Sigma K)^{D}=\Sigma L$ and $L=KG_1^{-1}G_2$, where $G_1$ and $G_2$ are given by \eqref{G1G2G3G4eq}.
\end{enumerate}

\end{theorem}

\begin{proof}
\eqref{mDMPfurperitem1}. It is clear by  \eqref{MDMPeADAm} and  Lemma \ref{mMPprole}\eqref{AAmMPeq}. \par
\eqref{mDMPfurperitem2}.
From \eqref{MDMPeADAm} we have that
\begin{equation*}
(A^{D,\mathfrak{m}})^2=A^{D}AA^{\mathfrak{m}}A^{D}AA^{\mathfrak{m}} =A^{D}AA^{\mathfrak{m}}AA^{D}A^{\mathfrak{m}}=A^{D}A^{\mathfrak{m}}.
\end{equation*}
Then, for an  even number $l$, it follows that
\begin{equation}\label{evennumbereq}
(A^{D,\mathfrak{m}})^{l}=((A^{D,\mathfrak{m}})^{2})^{\frac{l}{2}} =(A^{D}A^{\mathfrak{m}})^{\frac{l}{2}}.
\end{equation}
Moreover, if $l$ is  odd, then from \eqref{evennumbereq} and \eqref{MDMPeADAm},   we get that
\begin{align*}
(A^{D,\mathfrak{m}})^{l}=A^{D,\mathfrak{m}}(A^{D,\mathfrak{m}})^{l-1} =A^{D,\mathfrak{m}}(A^{D}A^{\mathfrak{m}})^{\frac{l-1}{2}} =AA^DA^{\mathfrak{m}}(A^{D}A^{\mathfrak{m}})^{\frac{l-1}{2}} =A(A^{D}A^{\mathfrak{m}})^{\frac{l+1}{2}}.
\end{align*}
\par
\eqref{mDMPfurperitem3}.  Using the Cline's Formula, i.e., $(XY)^D=X\left((YX)^D\right)^2Y$  for $X\in\mathbb{C}^{m\times n}$ and $Y\in\mathbb{C}^{n\times m}$,  from \eqref{MDMPeADAm} we infer that
\begin{equation*}
(A^2A^{\mathfrak{m}})^{D}= (A(AA^{\mathfrak{m}}))^{D}= A\left((AA^{\mathfrak{m}}A)^D\right)^2AA^{\mathfrak{m}} =A(A^D)^2AA^{\mathfrak{m}}=A^DAA^{\mathfrak{m}}=A^{D,\mathfrak{m}}.
\end{equation*}
\par
\eqref{mDMPfurperitem4}.
Using again Cline's Formula, from  \eqref{MDMPeADAm},  Lemma \ref{mMPprole}\eqref{AAmMPeq} and  Lemma \ref{Dinverseperle}\eqref{Dinverseperitem1},  we have
\begin{align*}
(A^{D,\mathfrak{m}})^D&=(A^D(AA^{\mathfrak{m}}))^D = A^D\left((AA^{\mathfrak{m}}A^D)^D\right)^2AA^{\mathfrak{m}}\nonumber\\&= A^D\left((A^D)^D\right)^2AA^{\mathfrak{m}}=(A^D)^DAA^{\mathfrak{m}}=(A^D)^{\#}AA^{\mathfrak{m}}.
\end{align*}
Then, again by  Cline's Formula, we have
\begin{align*}
((A^{D,\mathfrak{m}})^D)^D&=((A^D)^{\#}AA^{\mathfrak{m}})^D= (A^D)^{\#}\left((AA^{\mathfrak{m}}(A^D)^{\#})^D\right)^2AA^{\mathfrak{m}}\\
&=(A^D)^{\#}\left(((A^D)^{\#})^{\#}\right)^2AA^{\mathfrak{m}}=A^DAA^{\mathfrak{m}}.
\end{align*}
\par
\eqref{mDMPfurperitem7}. According to \eqref{MDMPeADAm},  Lemma \ref{mMPprole}\eqref{AAmMPeq} and  Lemma \ref{Dinverseperle}\eqref{Dinverseperitem2}, we see that
\begin{align*}
AA^{D,\mathfrak{m}}=A^{D,\mathfrak{m}}A & \Leftrightarrow AA^{D}(AA^{\mathfrak{m}}-I_n)= 0 \Leftrightarrow \mathcal{N}(A^{\sim})\subseteq \mathcal{N}(A^{k})\\
& \Leftrightarrow A^{D}(AA^{\mathfrak{m}}-I_n)= 0  \Leftrightarrow A^{D,\mathfrak{m}}= A^{D}.
\end{align*}
\par
\eqref{mDMPfurperitem5}.  Using  \eqref{minDMPHSdecGeq01}  and \cite[Theorem 2.5]{DMPorignre}, i.e., $   A^{D,\dag}=U
   \left(
     \begin{array}{cc}
       (\Sigma K)^D & 0 \\
       0 & 0 \\
     \end{array}
   \right)
    U^*$,  we have
that
\begin{equation*}
A^{D,\mathfrak{m}}=0\Leftrightarrow(\Sigma K)^D=0\Leftrightarrow A^{D,\dag}=0,
\end{equation*}
which, together with \cite[Proposition 2.14 (g)]{DMPorignre}, i.e., $A^{D,\dag}=0\Leftrightarrow A $ is nilpotent,  shows the item \eqref{mDMPfurperitem5} holds.
\par
\eqref{mDMPfurperitem6}.    It is obvious by  \eqref{HSAdec} and \eqref{minDMPHSdecGeq}.
\end{proof}

\begin{remark}
Under the hypotheses of Theorem \ref{mDMPfurperth},  uisng \eqref{HSAdec}   and \eqref{minMPdec01}, by direct calculation we have
\begin{equation*}
 A^2A^{\mathfrak{m}}=U
\left(
  \begin{array}{cc}
    \Sigma K G_1^{-1} & 0 \\
    0 & 0 \\
  \end{array}
\right)
U^*G.
\end{equation*}
Then,  using \eqref{minDMPHSdecGeq01}  in Theorem \ref{MDMPHSdecth},  we can verify that $A^{D,\mathfrak{m}}$  satisfies the definition of the Drazin inverse of $ A^2A^{\mathfrak{m}}$, i.e., $A^{D,\mathfrak{m}}= (A^2A^{\mathfrak{m}})^{D}$. Thus,   we succeed in avoiding the use of Cline's formula for the proof of  Theorem \ref{mDMPfurperth}\eqref{mDMPfurperitem3}.
\end{remark}

\section{Further characterizations and representations of the $\mathfrak{m}$-DMP inverse }\label{charepminDMPsec}
We shall  continue to characterize and represent the $\mathfrak{m}$-DMP inverse from different views in this section. We begin this section by  characterizing the  $\mathfrak{m}$-DMP inverse based on  its  essential  properties obtained in Section \ref{deminDMPsec}.

\begin{theorem}
Let  $A\in\mathbb{C}^{n\times n}_{k}$ with ${\rm rank}(A^{\sim}AA^{\sim})={\rm rank}(A)$, and let $X\in\mathbb{C}^{n\times n}$. Then the following statements are equivalent:
\begin{enumerate}[$(1)$]
  \item\label{MDMPchaitem1} $X=A^{D,\mathfrak{m}}$;
    \item\label{MDMPchaitem4} $\mathcal{R}(X)\subseteq\mathcal{R}(A^k)$, $A^DX=A^DA^{\mathfrak{m}}$;
  \item\label{MDMPchaitem2} $\mathcal{R}(X)\subseteq\mathcal{R}(A^k)$, $A^kX=A^kA^{\mathfrak{m}}$;
  \item\label{MDMPchaitem10} $\mathcal{N}(A^kA^{\mathfrak{m}})\subseteq\mathcal{N}(X)$, $XA=AA^D$;
  \item\label{MDMPchaitem8} $\mathcal{N}(A^kA^{\mathfrak{m}})\subseteq\mathcal{N}(X)$, $XA^{k+1}=A^k$.
\end{enumerate}
\end{theorem}

\begin{proof}
\eqref{MDMPchaitem1} $\Rightarrow$ \eqref{MDMPchaitem4}.  It is a direct corollary of Theorem \ref{MDMPATSth}\eqref{minDMPperitem2} and \eqref{MDMPeADAm}.
\par
\eqref{MDMPchaitem4} $\Rightarrow$ \eqref{MDMPchaitem2}. It is clear by premultiplying $A^DX=A^DA^{\mathfrak{m}}$ with $A^{k+1}$.
\par
\eqref{MDMPchaitem2} $\Rightarrow$ \eqref{MDMPchaitem10}.
Since ${\rm rank}(X)\leq {\rm rank}(A^k)$ by $\mathcal{R}(X)\subseteq\mathcal{R}(A^k)$, it follows from \eqref{rankAKAmerankAkeq} and $A^kX=A^kA^{\mathfrak{m}}$ that  ${\rm rank}(X)={\rm rank}(A^k)$ and $\mathcal{N}(X)= \mathcal{N}(A^kA^{\mathfrak{m}})$. Then, from Lemma \ref{Dinverseperle}\eqref{Dinverseperitem2} and $\mathcal{R}(X)\subseteq\mathcal{R}(A^k)$, we infer that
\begin{align*}
A^kX=A^kA^{\mathfrak{m}} &\Rightarrow (A^D)^kA^kXA=(A^D)^kA^kA^{\mathfrak{m}}A \\
 &\Rightarrow   A^DAXA=A^DAA^{\mathfrak{m}}A   \\
 &\Rightarrow   XA=A^DA .
\end{align*}
\par
\eqref{MDMPchaitem10} $\Rightarrow$\eqref{MDMPchaitem8}.  It is obvious by postmultiplying $XA=AA^D$ with $A^{k}$.
\par
\eqref{MDMPchaitem8} $\Rightarrow$ \eqref{MDMPchaitem1}.  Since $\mathcal{R} (X) \subseteq \mathcal{R}(A^k)$ and ${\rm rank}(A^k) \leq {\rm rank}(X)$ from $XA^{k+1}=A^k$, by \eqref{rankAKAmerankAkeq} and $\mathcal{N}(A^kA^{\mathfrak{m}})\subseteq\mathcal{N}(X)$ we have that
$\mathcal{R} (X) = \mathcal{R}(A^k)$ and $\mathcal{N}(A^kA^{\mathfrak{m}}) = \mathcal{N}(X)$, which implies that there exists $Y\in\mathbb{C}^{n\times n}$ such that
\begin{equation}\label{YAKAMEXeq}
YA^kA^{\mathfrak{m}}=X.
\end{equation}
 Then postmultiplying \eqref{YAKAMEXeq} with $A$ gives that $YA^k=XA$, implying $\mathcal{N}(A^k)\subseteq \mathcal{N}(XA)$. Consequently, by Lemma \ref{Dinverseperle}\eqref{Dinverseperitem2}, we get that
\begin{align*}
XA^{k+1}=A^k  & \Rightarrow XA^{k+1}(A^D)^kX=A^k (A^D)^kX \\
 & \Rightarrow XAA^DAX=A^DAX \\
 & \Rightarrow  XAX=X,
\end{align*}
implying $X=A^{(2)}_{\mathcal{R}(A^{k}),\mathcal{N}(A^kA^{\mathfrak{m}})}$. Therefore,  $X=A^{D,\mathfrak{m}}$ by Theorem \ref{MDMPATSth}\eqref{minDMPperitem3}.
\end{proof}

\begin{theorem}\label{minDMPeqchasecth}
Let  $A\in\mathbb{C}^{n\times n}_{k}$ with ${\rm rank}(A^{\sim}AA^{\sim})={\rm rank}(A)$, and let $X\in\mathbb{C}^{n\times n}$. Then the following statements are equivalent:
\begin{enumerate}[$(1)$]
  \item\label{minDMPeqchasecitem1} $X=A^{D,\mathfrak{m}}$;
  \item\label{minDMPeqchasecitem4}   $AX^2=X$, $AX=A^2A^DA^{\mathfrak{m}}$;
  \item\label{minDMPeqchasecitem3}  $AX^2=X$, $AX=P_{\mathcal{R}(A^{k}),\mathcal{N}(A^DA^{\mathfrak{m}})}$;
  \item\label{minDMPeqchasecitem2}  $AX^2=X$, $A^kX=A^kA^{\mathfrak{m}}$.
\end{enumerate}
\end{theorem}

\begin{proof}
 \eqref{minDMPeqchasecitem1}  $\Rightarrow$  \eqref{minDMPeqchasecitem4}. It is easily obtained by \eqref{MDMPeADAm}.
 \par
  \eqref{minDMPeqchasecitem4}  $\Rightarrow$  \eqref{minDMPeqchasecitem3}.
  According to \eqref{MDMPeADAm} and Theorem \ref{MDMPATSth}\eqref{minDMPperitem4}, we have that
  \begin{equation}\label{A2ADAmeqPeq}
   AX=A^2A^DA^{\mathfrak{m}}=AA^{D,\mathfrak{m}}= P_{\mathcal{R}(A^{k}),\mathcal{N}(A^DA^{\mathfrak{m}})}.
  \end{equation}
   \par
  \eqref{minDMPeqchasecitem3}  $\Rightarrow$  \eqref{minDMPeqchasecitem2}.
  Using \eqref{A2ADAmeqPeq}, from  $AX=P_{\mathcal{R}(A^{k}),\mathcal{N}(A^DA^{\mathfrak{m}})}$ we have that
  \begin{equation*}
    A^kX=A^{k-1}P_{\mathcal{R}(A^{k}),\mathcal{N}(A^DA^{\mathfrak{m}})}=A^{k-1}A^2A^DA^{\mathfrak{m}}=A^kA^{\mathfrak{m}}.
  \end{equation*}
 \par
  \eqref{minDMPeqchasecitem2}  $\Rightarrow$  \eqref{minDMPeqchasecitem1}. It follows from   $AX^2=X$ that
  \begin{equation*}
    X=AXX=AAX^2X=A^2XX^2=A^2AX^2X^2=A^3XX^3=...=A^kX^{k+1},
  \end{equation*}
which implies  $\mathcal{R}({X})\subseteq \mathcal{R}(A^k)$. Then,  by Lemma \ref{Dinverseperle}\eqref{Dinverseperitem2} and \eqref{MDMPeADAm}  we conclude that
\begin{align*}
  A^kX=A^kA^{\mathfrak{m}} & \Rightarrow    (A^D)^kA^kX=(A^D)^kA^kA^{\mathfrak{m}} \\  & \Rightarrow AA^DX=A^D AA^{\mathfrak{m}}\\  & \Rightarrow X=A^DAA^{\mathfrak{m}}= A^{D,\mathfrak{m}}.
\end{align*}
This completes the proof.
\end{proof}

Zuo et al.  in  \cite[Theorem 3.8]{diffDMPre} gave an interesting result of the DMP inverse, that is, for $A\in\mathbb{C}^{n\times n}_{k}$,
\begin{equation}\label{interDMPpreeq}
A^{D,\dag}=AA^{\dag}(I_n-\overline{A}AA^{\dag})^{D} =(I_n-\overline{A}AA^{\dag})^{D}AA^{\dag},
\end{equation}
where $\overline{A}=I_n-A$.
 The following theorem turns out analogous expressions of the $\mathfrak{m}$-DMP inverse.

\begin{theorem}
Let  $A\in\mathbb{C}^{n\times n}_{k}$ with ${\rm rank}(A^{\sim}AA^{\sim})={\rm rank}(A)$, and let $\overline{A}=I_n-A$. Then,
\begin{align}
A^{D,\mathfrak{m}}&=AA^{\mathfrak{m}}(I_n-\overline{A}AA^{\mathfrak{m}})^{D}\label{AAmIAAAmD01}\\
&=(I_n-\overline{A}AA^{\mathfrak{m}})^{D}AA^{\mathfrak{m}}\label{AAmIAAAmD02}.
\end{align}
\end{theorem}

\begin{proof}
Using \cite[Corollary 1]{Drainzeinre}, i.e., $(X+Y)^{D}=X^D+Y^D$, where $X,Y\in\mathbb{C}^{n\times n}$ satisfies $XY=YX=0$, and a clear fact \begin{equation*}
  (I_n-AA^{\mathfrak{m}})A^2A^{\mathfrak{m}}= A^2A^{\mathfrak{m}}(I_n-AA^{\mathfrak{m}})=0,
\end{equation*}
we can directly have that
\begin{equation*}
(I_n-AA^{\mathfrak{m}}+A^2A^{\mathfrak{m}})^{D} =(I_n-AA^{\mathfrak{m}})^{D}+(A^2A^{\mathfrak{m}})^{D}
=I_n-AA^{\mathfrak{m}}+(A^2A^{\mathfrak{m}})^{D}.
\end{equation*}
Hence, it follows from Theorem \ref{mDMPfurperth}\eqref{mDMPfurperitem3}, Theorem \ref{MDMPATSth}\eqref{minDMPperitem2}  and Lemma \ref{mMPprole}\eqref{AAmMPeq}  that
\begin{align*}
AA^{\mathfrak{m}}(I_n-\overline{A}AA^{\mathfrak{m}})^{D}&= AA^{\mathfrak{m}}(I_n-AA^{\mathfrak{m}}+A^2A^{\mathfrak{m}})^{D}\\
&=AA^{\mathfrak{m}}(I_n-AA^{\mathfrak{m}})+AA^{\mathfrak{m}}(A^2A^{\mathfrak{m}})^{D}\\
&=AA^{\mathfrak{m}}A^{D,\mathfrak{m}}= A^{D,\mathfrak{m}},
\end{align*}
and
\begin{align*}
(I_n-\overline{A}AA^{\mathfrak{m}})^{D}AA^{\mathfrak{m}}&= (I_n-AA^{\mathfrak{m}}+A^2A^{\mathfrak{m}})^{D}AA^{\mathfrak{m}}\\
&=(I_n-AA^{\mathfrak{m}})AA^{\mathfrak{m}}+(A^2A^{\mathfrak{m}})^{D}AA^{\mathfrak{m}}\\
&=A^{D,\mathfrak{m}}AA^{\mathfrak{m}}= A^{D,\mathfrak{m}},
\end{align*}
which show that \eqref{AAmIAAAmD01}  and \eqref{AAmIAAAmD02} are true.
\end{proof}

The full-rank factorization is also confirmed     as  a  powerful tool to study the generalized inverses.  And,   using the full-rank factorization, Cline  \cite{Drazinifurafre} and Zekraoui et al.  \cite{algebraicMMP} expressed the Drazin inverse and the Minkowski inverse, respectively. Based on their work, we present  a new representation of the $\mathfrak{m}$-DMP inverse in the following theorem.

\begin{theorem}\label{minDMPfulldth}
Let  $A\in\mathbb{C}^{n\times n}_k$ with ${\rm rank}(A^{\sim}AA^{\sim})={\rm rank}(A)>0$ and $A^k\neq 0$.
 Let
\begin{equation}\label{AfzBkCk}
\left.
  \begin{array}{ccccc}
   A=B_1C_1, & C_1B_1=B_2C_2, &C_2B_2=B_3C_3, & ..., &C_{k-1}B_{k-1}=B_kC_k, \\
  \end{array}
\right.
\end{equation}
be such that $B_1C_1$ is a full-rank factorization  of $A$, and $B_{i+1}C_{i+1}$ are  full-rank factorizations  of $C_{i}B_{i}$ $(i=1,2,...,k-1)$. Then
\begin{equation*}
  A^{D,\mathfrak{m}}= B_1 \cdots B_k(C_kB_k)^{-k}C_k \cdots C_2(B_1^{\sim}B_1)^{-1}B_1^{\sim}.
\end{equation*}
\end{theorem}

\begin{proof}
First,  it follows from \eqref{AfzBkCk}  that
\begin{equation*}
    C_k\cdots C_3C_2C_1B_1 = C_k\cdots C_3C_2B_2C_2 = \cdots = C_kB_kC_k\cdots C_2.
\end{equation*}
Then,  applying \cite[Formula (1.11)]{Drazinifurafre}, i.e., $A^D= B_1 \cdots B_k(C_kB_k)^{-(k+1)}C_k \cdots C_1$, and \cite[Theorem 8]{algebraicMMP}, i.e., $A^{\mathfrak{m}}=C_1^{\sim}(C_1C_1^{\sim})^{-1}(B_1^{\sim}B_1)^{-1}B_1^{\sim}$, to \eqref{MDMPeADAm},  we have that
\begin{align*}
 A^{D,\mathfrak{m}} &=
 B_1 \cdots B_k(C_kB_k)^{-(k+1)}C_k \cdots C_1B_1C_1C_1^{\sim}(C_1C_1^{\sim})^{-1}(B_1^{\sim}B_1)^{-1}B_1^{\sim} \\
   & = B_1 \cdots B_k(C_kB_k)^{-k}C_k \cdots C_2(B_1^{\sim}B_1)^{-1}B_1^{\sim},
\end{align*}
which completes the proof.
\end{proof}

In terms of  the full-rank decomposition,  Zhou and Chen \cite{DMPfutre10} derived  integral representations  of the DMP inverse,  which do not require the  restriction for the spectrum of a   matrix. And,  K{\i}l{\i}\c{c}man  et al. \cite{WminMPorire} obtained an integral representation  of the weighted Minkowski inverse.  Motivated by their work, we show an  integral representation of the $\mathfrak{m}$-DMP inverse as follows.

\begin{theorem}
Let  $A\in\mathbb{C}^{n\times n}_k$ with ${\rm rank}(A^{\sim}AA^{\sim})={\rm rank}(A)>0$ and $A^k\neq 0$, and let  the full-rank factorization of $A$ be as in \eqref{AfzBkCk}. Then,
\begin{equation}\label{minDMPinterpreth}
 A^{D,\mathfrak{m}} =  \int_{\rm{0}}^\infty  M{\rm exp}(-B_1^{\sim}B_1t)B_1^{\sim}dt,
\end{equation}
where $M=B_1 \cdots B_k(C_kB_k)^{-k}C_k \cdots C_2$.
\end{theorem}

\begin{proof}
We first claim that $B^{\mathfrak{m}}$ exists. In fact,
\begin{align*}
  {\rm rank}(B_1^{\sim}B_1) \leq {\rm rank}(B_1) & = {\rm rank}(A)  = {\rm rank}(A^{\sim}AA^{\sim})\\
 & = {\rm rank}(C_1^{\sim}B_1^{\sim}B_1C_1C_1^{\sim}B_1^{\sim})\leq {\rm rank}(B_1^{\sim}B_1).
\end{align*}
Then, ${\rm rank}(B_1) ={\rm rank}(B_1^{\sim}B_1)= {\rm rank}(B_1^{\sim}B_1B_1^{\sim})$  since $B_1^{\sim}$ is of full row rank. Thus, $B^{\mathfrak{m}}$ exists by Lemma \ref{HSmMpdecth}\eqref{minMPexsitconditem3}. Furthermore, it follows from Lemma \ref{mMPprole}\eqref{mMPRN}  and Theorem \ref{MDMPATSth}\eqref{minDMPperitem2} that
\begin{equation*}
\mathcal{N}(B_1^{\sim})\subseteq \mathcal{N}(C_1^{\sim}B_1^{\sim}) = \mathcal{N}(A^{\sim}) = \mathcal{N}(A^{\mathfrak{m}})\subseteq\mathcal{N}({A^kA^{\mathfrak{m}}}) =\mathcal{N}(A^{D,\mathfrak{m}}),
\end{equation*}
which, together with Lemma \ref{mMPprole}\eqref{AAmMPeq}, shows that
\begin{equation}\label{mDMPB1mDMPeq}
A^{D,\mathfrak{m}}B_1B_1^{\mathfrak{m}}=A^{D,\mathfrak{m}}.
\end{equation}
Finally, applying Theorem \ref{minDMPfulldth} and \cite[Corollary 8]{WminMPorire}, i.e., $X^{\mathfrak{m}}=\int_{\rm{0}}^\infty {\rm exp}(-X^{\sim}Xt)X^{\sim}dt$ for $X\in\mathbb{C}^{n\times n}$, to \eqref{mDMPB1mDMPeq} gives \eqref{minDMPinterpreth} immediately.
\end{proof}

It is well known   that the Moore-Penrose inverse of  $A\in\mathbb{C}^{m\times n}$ can be expressed as a limit \cite{benlimitre}, i.e.,
\begin{equation*}
  A^{\dag}={\rm lim }_{\lambda  \to 0} (\lambda I_n+A^*A)^{-1}A^*.
\end{equation*}
And, it has always been a hot topic to compute the generalized inverses  by means of  the limiting process.  Ma et al. \cite{chaDMPmare} and K{\i}l{\i}\c{c}man et al. \cite{WminMPorire} presented a few limiting expressions of the DMP inverse and    weighted Minkowski inverse, respectively. The next theorem gives several  limit representations for the $\mathfrak{m}$-DMP inverse.

\begin{theorem}
Let  $A\in\mathbb{C}^{n\times n}_{k}$ with ${\rm rank}(A^{\sim}AA^{\sim})={\rm rank}(A)$.  Then,
\begin{align}
A^{D,\mathfrak{m}}&={\rm lim }_{\lambda  \to 0} A^{k}(\lambda I_n+A^{\mathfrak{m}}A^{k+1})^{-1}A^{\mathfrak{m}}\label{minDMPlimityuaneq01}\\
&={\rm lim }_{\lambda  \to 0}A^{k}A^{\mathfrak{m}}(\lambda I_n+A^{k+1}A^{\mathfrak{m}})^{-1}\label{minDMPlimitrepreseeq02}\\
&={\rm lim }_{\lambda  \to 0}(\lambda I_n+A^{k})^{-1}A^{k}A^{\mathfrak{m}}\label{minDMPlimitrepreseeq03}\\
&={\rm lim }_{\lambda  \to 0}(\lambda I_n+A^{k})^{-1}A^{k}(\lambda I_n+A^{\sim}A)^{-1}A^{\sim}\label{minDMPlimitrepreseeq04}.
\end{align}
\end{theorem}

\begin{proof}
Since $\mathcal{R}(A^kA^{\mathfrak{m}})=\mathcal{R}(A^k)$ from \eqref{rankAKAmerankAkeq},  using  Theorem \ref{MDMPATSth}\eqref{minDMPperitem3} we see that
\begin{equation}\label{MinDMPATS2othereq}
 A^{D,\mathfrak{m}}=A^{(2)}_{\mathcal{R}(A^{k}), \mathcal{N}(A^{k}A^{\mathfrak{m}})}=A^{(2)}_{\mathcal{R}(A^{k}A^{\mathfrak{m}}), \mathcal{N}(A^{k}A^{\mathfrak{m}})}.
\end{equation}
Thus, applying \eqref{ATS2YZlimteq},  \eqref{ATS2Ilimiteq02} and  \eqref{ATS2Ilimiteq03} to \eqref{MinDMPATS2othereq} yields \eqref{minDMPlimityuaneq01}, \eqref{minDMPlimitrepreseeq02} and \eqref{minDMPlimitrepreseeq03}, respectively. Then, substituting
\cite[Corollary 11]{WminMPorire}, i.e., $A^{\mathfrak{m}}={\rm lim }_{\lambda  \to 0}  (\lambda I_n+A^{\sim}A)^{-1}A^{\sim}$, into \eqref{minDMPlimitrepreseeq03} shows  \eqref{minDMPlimitrepreseeq04} immediately.  This finishes the proof.
\end{proof}

\begin{example}
Let us test the matrix $A$ given in Example \ref{MinDMPmainExample}.  Then,  $k:={\rm Ind}(A)=2$,
\begin{align*}
B:&=A^{k}(\lambda I_n+A^{\mathfrak{m}}A^{k+1})^{-1}A^{\mathfrak{m}}=
A^{k}A^{\mathfrak{m}}(\lambda I_n+A^{k+1}A^{\mathfrak{m}})^{-1}\\&=
(\lambda I_n+A^{k})^{-1}A^{k}A^{\mathfrak{m}}=
\left(
  \begin{array}{ccccc}
    \frac{2}{\lambda+1} & \frac{-1}{\lambda+1} & \frac{1}{\lambda+1}& 0 & 0 \\
  \frac{2}{\lambda+1} & \frac{-1}{\lambda+1} & \frac{1}{\lambda+1}& 0 & 0 \\
   0 & 0 & 0 & 0 & 0    \\
   0 & 0 & 0 & 0 & 0    \\
   0 & 0 & 0 & 0 & 0 \\
  \end{array}
\right),\\
C:&=(\lambda I_n+A^{k})^{-1}A^{k}(\lambda I_n+A^{\sim}A)^{-1}A^{\sim}=
\left(
  \begin{array}{ccccc}
    \frac{\lambda+2}{(\lambda+1)^3} & \frac{-\lambda-1}{(\lambda+1)^3} & \frac{1}{(\lambda+1)^3} & 0 & 0 \\
   \frac{\lambda+2}{(\lambda+1)^3} & \frac{-\lambda-1}{(\lambda+1)^3} & \frac{1}{(\lambda+1)^3} & 0 & 0 \\
   0 & 0 & 0 & 0 & 0    \\
   0 & 0 & 0 & 0 & 0    \\
   0 & 0 & 0 & 0 & 0 \\
  \end{array}
\right).
\end{align*}
It is easy to check that ${\rm lim }_{\lambda  \to 0}B={\rm lim }_{\lambda  \to 0}C=A^{D,\mathfrak{m}}$, where $A^{D,\mathfrak{m}}$ has been shown  in Example \ref{MinDMPmainExample} and so  is omitted.
\end{example}

\section{Applications of the $\mathfrak{m}$-DMP inverse in solving some equations}\label{appminDMPsec}

Our motivation in this section arises mainly from the work that Ma el al.  \cite{chaDMPmare} solved singular linear  systems by using DMP inverse, and gave a condensed Cramer's  rule for computing  the DMP-inverse solution.  We start with  a consideration on a system of linear equations.

\begin{theorem}\label{appequth1}
 Let  $A\in\mathbb{C}^{n\times n}_{k}$ with ${\rm rank}(A^{\sim}AA^{\sim})={\rm rank}(A)$ and $b\in\mathbb{C}^{n}$, and let a system of linear equations be
 \begin{equation}\label{sysequations1}
 A^{k}x=A^kA^{\mathfrak{m}}b.
 \end{equation}
 Then the general solution of  the system \eqref{sysequations1} is
 \begin{equation}\label{minDMPappeqgesoeq}
 x=A^{D,\mathfrak{m}}b+(I_n-A^{D,\mathfrak{m}}A)v,
 \end{equation}
where arbitrary  $v\in\mathbb{C}^n$. Moreover,
 \begin{equation*}
 x=A^{D,\mathfrak{m}}b
 \end{equation*}
 is the unique solution to the system \eqref{sysequations1}  on  $\mathcal{R}(A^k)$.
\end{theorem}

\begin{proof}
It is clear  by \eqref{MDMPeADAm} that $A^{D,\mathfrak{m}}b$ is a solution to \eqref{sysequations1}. Hence,  using Theorem \ref{MDMPATSth}\eqref{minDMPperitem5}, we have that the set of
 all solutions of \eqref{sysequations1} is
\begin{equation*}
\left\{
  \begin{array}{c|c}
    A^{D,\mathfrak{m}}b + \alpha & \alpha \in \mathcal{N}(A^k)  \\
  \end{array}
\right\}=
\left\{
  \begin{array}{c|c}
    A^{D,\mathfrak{m}}b + \alpha & \alpha \in \mathcal{R}(I_n-A^{D,\mathfrak{m}}A) \\
  \end{array}
\right\},
\end{equation*}
which shows that the general solution of \eqref{sysequations1} is \eqref{minDMPappeqgesoeq}. Moreover, since $\mathcal{R}(A^k) \oplus \mathcal{N}(A^k)= \mathbb{C}^{n}$ by ${\rm Ind}(A)=k$, using  Theorem \ref{MDMPATSth}\eqref{minDMPperitem2} we see that $A^{D,\mathfrak{m}}b\in\mathcal{R}(A^k)$ is the unique solution to  \eqref{sysequations1}  on  $\mathcal{R}(A^k)$.
\end{proof}

Wang et al.   \cite{MCore} considered an  interesting  least-squares problem in Frobenius norm,  that is,
\begin{equation*}
     \Vert (AA^{\dag})^{\sim}Ax-b  \Vert_{F} ={\rm min} \text{ subject to }x\in\mathcal{R}(A),
\end{equation*}
where $A\in\mathbb{C}^{n\times n}_{1}$ with ${\rm rank}(A^{\sim}A)={\rm rank}(A)<n$, and $b \in\mathbb{C}^{n}$. In the following theorem, we  discuss  an analogous optimization problem on the $\mathfrak{m}$-DMP inverse.

\begin{theorem}\label{minFnormth}
 Let  $A\in\mathbb{C}^{n\times n}_{k}$ be given by \eqref{HSAdec} with $r:={\rm rank}(A^{\sim}AA^{\sim})={\rm rank}(A)<n$. Let  $b\in\mathbb{C}^{n}$ be
 $b=GU\left(\begin{array}{c}  b_1 \\ b_2\\\end{array}\right)$, where $b_1\in\mathbb{C}^{r}$ satisfies  $G_1^{-1}b_1\in\mathcal{R}\left((\Sigma K)^{D}\right)$ and $G_1$ is given by \eqref{G1G2G3G4eq},  and $b_2\in\mathbb{C}^{n-r}$.
 Then
 \begin{equation}\label{mineq}
 \mathop {\min }\limits_{x\in\mathcal{R}({A^{k}})}  \Vert (AA^{\dag})^{\sim}Ax-b  \Vert_{F}=\left\Vert b_2 \right\Vert_{F}.
 \end{equation}
 Moreover,
  \begin{equation*}
 x=A^{D,\mathfrak{m}}b
 \end{equation*}
is the unique solution of \eqref{mineq}.
\end{theorem}

\begin{proof}
For every $x\in\mathcal{R}(A^k)$, it follows from Lemma \ref{Dinverseperle}\eqref{Dinverseperitem1} that there exits $y\in\mathbb{C}^{n}$ such that $x=A^Dy$.  Put
$y=U\left(
      \begin{array}{c}
        y_1 \\
        y_2 \\
      \end{array}
    \right)
$,
where $y_1\in\mathbb{C}^{r}$ and $y_2\in\mathbb{C}^{n-r}$.
Using \cite[Formula 2.2]{Coreinverse}, i.e.,
$AA^{\dag}=U\left(\begin{array}{cc} I_r & 0 \\0 & 0 \\\end{array}\right)U^*$, from \eqref{HSAdec}, \eqref{G1G2G3G4eq} and \eqref{Ddec01}  we infer  that
\begin{align*}
\left\Vert(AA^{\dag})^{\sim}Ax-b   \right\Vert_{F}
=&\left\Vert(AA^{\dag})^{\sim}AA^Dy-b   \right\Vert_{F}
\\
=&
\left\Vert GU
\left(\begin{array}{cc} I_r & 0 \\0 & 0 \\\end{array}\right)
 \left(\begin{array}{cc}G_1 & G_2 \\ G_2^* & G_4 \\\end{array}\right)
 \left(\begin{array}{cc}\Sigma K & \Sigma L \\ 0 & 0 \\\end{array}\right) \right. \\
& ~~\left. \left(\begin{array}{cc}(\Sigma K)^{D} & \left((\Sigma K)^D\right)^2\Sigma L \\0 & 0 \\\end{array}\right) \left(\begin{array}{c}y_1\\y_2\\\end{array}\right)- GU \left(\begin{array}{c}b_1\\b_2\\\end{array}\right)
 \right\Vert_{F} \\
= & \left\Vert \left(\begin{array}{c}G_1(\Sigma K)^D\Sigma K y_1 +G_1 (\Sigma K)^D\Sigma L y_2 -b_1\\- b_2\\\end{array}\right)
  \right\Vert_{F} \\
  =&\left( \left\Vert G_1(\Sigma K)^D\Sigma K y_1 +G_1 (\Sigma K)^D\Sigma L y_2 -b_1 \right\Vert_{F}^2 + \left\Vert b_2 \right\Vert_{F}^2 \right)^{\frac{1}{2}} \geq  \left\Vert b_2 \right\Vert_{F}.
\end{align*}
Since $\Sigma K(\Sigma K)^DG_1^{-1}b_1=G_1^{-1}b_1$ by the condition  $G_1^{-1}b_1\in\mathcal{R}\left((\Sigma K)^{D}\right)$,  we have that if
\begin{equation}\label{minDMPappy1eq}
  y_1=(\Sigma K)^D\Sigma Ly_2-G_1^{-1}b_1 \text{ and } y_2\in\mathbb{C}^{n-r},
\end{equation}
then $G_1(\Sigma K)^D\Sigma K y_1 +G_1 (\Sigma K)^D\Sigma L y_2 -b_1 =0$, which implies that $\left\Vert(AA^{\dag})^{\sim}Ax-b   \right\Vert_{F}$  assumes  the  minimum value,
\begin{equation*}
 \mathop {\min }\limits_{x\in\mathcal{R}({A^{k}})}  \Vert (AA^{\dag})^{\sim}Ax-b  \Vert_{F}= \left\Vert b_2 \right\Vert_{F}.
\end{equation*}
Therefore, by \eqref{Ddec01},  \eqref{minDMPappy1eq} and \eqref{minDMPHSdecGeq01},   it follows that
\begin{align*}
 x&=A^Dy= U \left(\begin{array}{cc}(\Sigma  K)^D & \left((\Sigma K)^D\right)^2\Sigma L \\0 & 0 \\\end{array}\right) \left(\begin{array}{c}y_1\\y_2\\\end{array}\right)\\
&=U \left(\begin{array}{c}(\Sigma K)^D(-(\Sigma K)^D\Sigma Ly_2-G_1^{-1}b_1)+ \left((\Sigma K)^D\right)^2\Sigma L y_2\\0  \\\end{array}\right)\\
&=U\left(\begin{array}{c}(\Sigma K)^DG_1^{-1}b_1\\0  \\\end{array}\right)=A^{D,\mathfrak{m}}b,
\end{align*}
which shows that $ x=A^{D,\mathfrak{m}}b$ is the unique solution of \eqref{mineq}.
\end{proof}

\begin{remark}
 Let  $A\in\mathbb{C}^{n\times n}_{k}$ be given by \eqref{HSAdec} with $r:={\rm rank}(A)< n$, and let  $b\in\mathbb{C}^{n}$ be
 $b=U\left(\begin{array}{c}  b_1 \\ b_2\\\end{array}\right)$, where $b_1\in\mathbb{C}^{r}$ is  such that $b_1\in\mathcal{R}\left((\Sigma K)^{D}\right)$,  and $b_2\in\mathbb{C}^{n-r}$. In terms of the same argument  in Theorem \ref{minFnormth}, we have that
 \begin{equation}\label{DMPappeq}
\mathop {\min }\limits_{x\in\mathcal{R}({A^{k}})}  \Vert Ax-b  \Vert_{F}=\left\Vert b_2 \right\Vert_{F}.
\end{equation}
Futhermore, $ x=A^{D,\dag}b$
is the unique solution of \eqref{DMPappeq}.
\end{remark}

We end up this section with presenting a condensed Cramer's  rule to directly calculate  the unique solution of \eqref{sysequations1}    and \eqref{mineq}.  Let  the determinant of $A\in\mathbb{C}^{ n \times n}$  be  ${\rm det}(A)$, and by $A(i\to b)$ we denote a matrix obtained by replacing the $i$th column of $A\in\mathbb{C}^{ n \times n}$ with $b\in\mathbb{C}^{n}$.

\begin{theorem}\label{MinDMPcramerth}
Let  $A\in\mathbb{C}^{n\times n}_{k}$ with ${\rm rank}(A^{\sim}AA^{\sim})={\rm rank}(A)$,  $b\in\mathbb{C}^n$, and   $t={\rm rank}(A^{k})$. Assume $V\in\mathbb{C}^{n\times (n-t)}$ and $W\in\mathbb{C}^{(n-t)\times n}$ are such that $\mathcal{R}(V)=\mathcal{N}(A^{k})$ and $\mathcal{N}(W)=\mathcal{R}(A^{k})$. Denote $E=V(WV)^{-1}W$. Then
  the components of  $x=A^{D,\mathfrak{m}}b$ are given by
\begin{equation}\label{applcomponenteq}
x_i=\frac{{\rm det}((A^k+E)(i\to A^kA^{\mathfrak{m}}b))}{{\rm det}(A^k+E)},~i=1,2,..,n.
\end{equation}
\end{theorem}

\begin{proof}
 Since \cite[Theorem 3.1]{chaDMPmare} has proved  that $E$ exists,  $\mathcal{N}(E)=\mathcal{R}(A^k)$, and $(A^k+E)^{-1}=(A^k)^{D}+E^{\#}$, from  Lemma \ref{Dinverseperle}\eqref{Dinverseperitem2} and  \eqref{MDMPeADAm} we see that
the system of linear equations
\begin{equation}\label{minDMPbtreq}
(A^k+E)x=A^kA^{\mathfrak{m}}b
\end{equation}
 has the unique solution
\begin{equation*}
x=(A^k+E)^{-1}A^kA^{\mathfrak{m}}b =(A^k)^{D}A^kA^{\mathfrak{m}}b+E^{\#}A^kA^{\mathfrak{m}}b =A^{D}AA^{\mathfrak{m}}b=A^{D,\mathfrak{m}}b.
\end{equation*}
Finally, applying Cramer's rule \cite{Cramerre}  to the nonsingular linear system \eqref{minDMPbtreq} gives \eqref{applcomponenteq} immediately.
\end{proof}

\begin{example}
Consider the matrix $A$ given in Example \ref{MinDMPmainExample}, and let
\begin{equation*}
 b=\left(
     \begin{array}{c}
0.15735\\0.15735\\0.1415\\-0.1\\-0.2\\
     \end{array}
   \right),
   V=\left(
     \begin{array}{cccc}
  0&0&0&0\\1&0&0&0\\0&1&0&0\\0&0&1&0\\0&0&0&1\\
     \end{array}
   \right),
   W=\left(
     \begin{array}{ccccc}
   0&0&1&0&0\\0&0&0&1&0\\0&0&0&0&1\\-1&1&0&0&0\\
     \end{array}
   \right).
\end{equation*}
Then, $k:={\rm Ind}(A)=2$,  and the Hartwig-Spindelb\"{o}ck decomposition of $A$ is
$A=U\left(
     \begin{array}{cc}
       \Sigma K & \Sigma L \\ 0 & 0 \\
     \end{array}
   \right)U^*$, where
\begin{align*}
U&=\left(
     \begin{array}{ccccc}
-0.40825&0.70711&0&0&0.57735\\-0.8165&0&0&0&-0.57735\\-0.40825&-0.70711&0&0&0.57735\\0&0&0&1&0\\0&0&1&0&0\\
     \end{array}
   \right),
\Sigma=\left(
     \begin{array}{cc}
1.7321&0\\0&1\\
     \end{array}
   \right),\\
K&=  \left(
     \begin{array}{cc}
0.28868&-0.5\\-0.28868&0.5\\
     \end{array}
   \right),
L=\left(
     \begin{array}{ccc}
0&-0.70711&-0.40825\\0&-0.70711&0.40825\\
     \end{array}
   \right).
\end{align*}
Furthermore,
\begin{align*}
  G_1&= \left(
     \begin{array}{cc}
-0.66667&-0.57735\\-0.57735&0\\
     \end{array}
   \right),
b_1=\left(
     \begin{array}{c}
 0.12201\\0.21132\\
     \end{array}
   \right),
   b_2= \left(
     \begin{array}{c}
0.2\\0.1\\0.1\\
     \end{array}
   \right), \\
   (\Sigma K)^{D}&=\left(
     \begin{array}{cc}
0.5&-0.86603\\-0.28868&0.5\\
     \end{array}
   \right),
   A^k=\left(
         \begin{array}{ccccc}
 1&0&0&0&0\\1&0&0&0&0\\0&0&0&0&0\\0&0&0&0&0\\0&0&0&0&0\\
         \end{array}
       \right),
 E=\left(
         \begin{array}{ccccc}
 0&0&0&0&0\\-1&1&0&0&0\\0&0&1&0&0\\0&0&0&1&0\\0&0&0&0&1\\
         \end{array}
       \right).
\end{align*}
Then it can easily be checked that ${\rm rank}(A^k)=1$, $G_1^{-1}b_1\in\mathcal{R}\left((\Sigma K)^{D}\right)$, $\mathcal{R}(V)=\mathcal{N}(A^{k})$ and $\mathcal{N}(W)=\mathcal{R}(A^{k})$. Using Theorems \ref{appequth1} and \ref{minFnormth}, we have that the unique solution of  the system \eqref{sysequations1}  on  $\mathcal{R}(A^k)$ and  the system \eqref{mineq} is
\begin{equation}\label{AmDMPbexample}
 x =A^{D,\mathfrak{m}}b= \left(
                          \begin{array}{ccccc}
   0.29885&0.29885&0&0&0\\
                          \end{array}
                        \right)^*,
\end{equation}
and
\begin{equation*}
  \mathop {\min }\limits_{x\in\mathcal{R}({A^{k}})}  \Vert (AA^{\dag})^{\sim}Ax-b  \Vert_{F}=\left\Vert b_2 \right\Vert_{F}=0.24495.
\end{equation*}
And, it is easy to check that the solution $x$ calculated by  \eqref{applcomponenteq} in Theorem \ref{MinDMPcramerth} is equal to $x$ given in \eqref{AmDMPbexample}.
\end{example}

\section{Conclusion}\label{conclusionsec}
This paper  defines  the $\mathfrak{m}$-DMP inverse in Minkowski space, and shows some of its properties, characterizations, representations,  and  applications in solving a system of linear equations and a constrained least norm problem.
\par
Not only because the $\mathfrak{m}$-DMP inverse, as a new generalized inverse, is an extension of the DMP inverse in Minkowski space, but also because of the wide research background of the DMP inverse, we are convinced that the $\mathfrak{m}$-DMP inverse  still has more  potential results and applications  to explore.
Several  future directions for the research of the $\mathfrak{m}$-DMP inverse can be described as follows:
\begin{itemize}
   \item[$(1)$] The perturbation analysis and  iterative methods for the  $\mathfrak{m}$-DMP inverse will  be two  topics  worth studying.

  \item[$(2)$] Generalizing generalized inverses by weighting is always an important part in studying generalized inverses. And, Meng \cite{DMPfutre7}  defined $W$-weighted DMP inverse of a rectangular matrix, which is a generalization of the DMP inverse of a square matrix.  It is  equally interesting to discuss the $\mathfrak{m}$-DMP inverse for rectangular matrices.

  \item[$(3)$] Inspired by  the work of \cite{DMPfutre21},  it is natural to ask what interesting  characterizations and applications  for the two new matrix classes $A^{D}AA^{\sim}$ and $A^{\sim}AA^{D}$ can be obtained.

\end{itemize}


\end{document}